\documentclass{amsart}

\usepackage{amsmath}
\usepackage[margin=3.5cm]{geometry}
\usepackage{amsthm}
\usepackage{amssymb}

\newtheorem{theorem}{Theorem}[section]
\newtheorem{lemma}[theorem]{Lemma}
\newtheorem{corollary}[theorem]{Corollary}

\theoremstyle{definition}
\newtheorem{dfn}[theorem]{Definition}
\newtheorem*{pr}{Proof of Theorem 3.1}
\theoremstyle{remark}
\newtheorem{remark}[theorem]{Remark}

\title{On the conditions of fixed-point theorems concerning $F$-contractions} 

\author{S\'andor Kaj\'ant\'o}
\author{Andor Luk\'acs}
\address{``Babe\c{s}-Bolyai'' University \\ Faculty of Mathematics and Computer Science \\
Kog\u{a}lniceanu Street Nr. 1 \\
400084, Cluj-Napoca \\
Romania}
\email{lukacs.andor@math.ubbcluj.ro}
\begin{document}

\begin{abstract}
We prove a fixed-point theorem that generalises and simplifies a number of results in the theory of $F$-contractions. We show that all of the previously imposed conditions on the operator can be either omitted or relaxed. Furthermore, our result is formulated in the more general context of $b$-metric spaces and $\varphi$-contractions. We also point out that the framework of $F$-contractions can be reformulated in an equivalent way that is both closer in spirit to the classical syntax of Banach-type fixed point theorems, and also more natural and easier to deal with in the proofs. 
\end{abstract}
\maketitle


\section{Introduction}
One of the most significant result of metric fixed-point theory is Banach's fixed-point theorem. 
Since this theorem has many relevant applications both in mathematics and in other sciences, it gave rise to a century-long research in the theory of fixed-points that is also flourishing nowadays. 
As a result, in the literature there are many generalizations of Banach's theorem. 
These generalizations usually follow two different approaches: on the one hand, the contractive condition can be relaxed, on the other, the metric space itself can be generalised. 
Following the first approach, in 2012 Wardowski in \cite{wardowski_fcontr_2012} proved a fixed-point theorem for a new type of operators on a complete metric space that had a remarkable impact on metric fixed-point theory. These new type of operators have been called by Wardowski $F$-contractions. They consist of those operators $T$ on a (complete) metric space that satisfy the contractive condition
\[Tx\ne Ty \mbox{ implies } \tau+F(d(Tx,Ty))\le F(d(x,y))\]
for some $\tau>0$ and some $F\colon(0,\infty)\to\mathbb{R}$ that satisfies (F1) $F$ is strictly increasing; (F2) $F(\alpha_n)\to-\infty$ if and only if $\alpha_n\to0$; (F3) $\lim_{x\to0^+}x^kF(x)=0$ for some $k\in(0,1)$.

The theory of $F$-contractions was recently developed further by many authors. 
An important guideline in some of these developments was the investigation of the necessity of the conditions (F1), (F2) and (F3). 
Secelean in \cite{secelean_2016} proved that (F3) can be dropped if we assume some boundedness condition on the operator $T$. 
Furthermore, if $F$ is continuous then (F3) can be dropped without any extra assumption on $T$.

In \cite{secelean_wardowski_2016} the authors introduced the notion of weak $\psi F$-contractions (a generalization of $F$-contractions using the idea of  classical $\varphi$-contractions introduced by Browder in \cite{browder_1968}). They generalised the previously discussed results of Secelean, by replacing (F1) with a weaker boun\-ded\-ness assumption on $F$ (and still requiring the boundedness assumption on $T$). 

On a different note, in \cite{kaj_luk_2017} the authors extended the original results of Wardowski to $b$-metric spaces without using condition (F2). 

In this paper we generalise and simplify both of the mentioned results in \cite{secelean_wardowski_2016} and \cite{kaj_luk_2017}. 
Our main result uses only a relaxed version of (F2) in the framework of generalised $\varphi$-contractions in $b$-metric spaces to conclude that our more general operators admit a unique fixed-point. 
The simplification lies on the observation that the framework of $F$-contractions can be put in a more natural context (closer in spirit to the original Banach fixed-point theorem and easier to deal with in the proofs).

\section{Preliminary definitions and results}
In this section we enlist those basic structures and terms that will be used to present our results. The first three definitions are classic, while the last two give a simplified version of Wardowski's $F$-contractions (\cite{wardowski_fcontr_2012}).

\begin{dfn}[\cite{bakhtin_1989,czerwik_1993}] We say that \emph{$(X,d)$ is a $b$-metric space with constant $s\ge1$} if $d\colon X\times X\to [0,\infty)$ satisfies the following conditions for every $x,y,z\in X$:
\begin{itemize}
		\item[(i)] $d(x,y)=0$ if and only if $x=y$;
		\item[(ii)] $d(x,y)=d(y,x)$;
		\item[(iii)] $d(x,y)\le s[d(x,z)+d(z,y)]$.
\end{itemize}
\end{dfn}
In general, the distance functional $d$ is not continuous. The following well-known lemma will help us to overcome this problem. 
\begin{lemma}[\cite{kirk_fixed_2014}]
	 If $(X,d)$ is a $b$-metric space with constant $s\ge1$, $x^*,y^*\in X$ and $(x_n)_{n\in\mathbb{N}}$ is a convergent sequence in $X$ with $x_n \to x^*$ then
    \[
    \frac{1}{s} d(x^*, y^*)\le \liminf_{n\to\infty} d(x_n, y^*)\le\limsup_{n\to\infty} d(x_n, y^*)\le sd(x^*,y^*).
    \]
    \label{lem:liminfsup}
\end{lemma}

\begin{proof}
	  If we apply twice the relaxed triangle inequality, we get for every $n\in \mathbb{N}$
    \[
    \frac{1}{s}d(x^*,y^*) - d(x_n, x^*)\le d(x_n, y^*)\le sd(x^*,y^*) + sd(x_n, x^*).
    \]
    If we take $\liminf$ on the left-hand side inequality and $\limsup$ on the right-hand side inequality, we obtain the desired property.
\end{proof}

\begin{dfn} Let $(X,d)$ be a complete $b$-metric space with constant $s\ge1$ and $T\colon X\to X$ an operator. We say that $T$ is \emph{a Picard-operator} if $T$ has a unique fixed point $x^*$ and $T^nx\to x^*$ for all $x\in X$.
\end{dfn}

The following definition is essentially that of comparison functions found in the literature (see ex. \cite{boyd_wong_1969,browder_1968,matkowski_1975,rus_1982}).
\begin{dfn} 
	Let $\Phi$ be the set of those functions $\varphi\colon(0,\infty)\to(0,\infty)$ that satisfy the following two conditions:
	\begin{itemize}
		\item [(i)] $\varphi$ is non-decreasing;
		\item [(ii)]  $\lim_{n\to\infty} \varphi^n(t)=0$ for every $t\in(0,\infty)$, where $\varphi^n$ means the $n$-fold composition of $\varphi$ with itself.
	\end{itemize}
\end{dfn}

\begin{dfn}
	We define two sets of functions $G\colon (0,\infty)\to(0,\infty)$ as follows.
	\begin{itemize}
		\item [(i)] $G\in\mathbb{G}_1$ if and only if $\inf G>0$.
		\item [(ii)] $G\in\mathbb{G}_2$ if and only if for every sequence $(\alpha_n)_{n\in\mathbb{N}}\subseteq(0,\infty)$ we have
		\[
			\lim_{n\to\infty} G(\alpha_n)=0\iff \lim_{n\to\infty}\alpha_n=0;
		\]
	\end{itemize}
\end{dfn}

\begin{dfn}
	Let $(X,d)$ be a $b$-metric space with constant $s\ge1$, and $G,\varphi\colon(0,\infty)\to(0,\infty)$ be two functions. We define the set $\mathbb{T}(X,G,\varphi)$ to consist of those operators $T\colon X\to X$ that for all $x,y\in X$ satisfy the following property:
	\begin{itemize}
		\item [(G)] $d(Tx, Ty)\ne 0 \mbox{ implies }
    G(d(Tx, Ty))\le \varphi (G(d(x,y))).$
	\end{itemize}
\end{dfn}

\section{Main result and implications}
Our main result is the following.
\begin{theorem} Let $(X,d)$ be a complete $b$-metric space with constant $s\ge1$. Suppose that there exists a $G\in\mathbb{G}_1\cup\mathbb{G}_2$ and $\varphi\in\Phi$ such that $T\in\mathbb{T}(X,G,\varphi)$. Then $T$ is a Picard-operator.
\label{thm:main}
\end{theorem}

We will prove this theorem in the next section. Before that we investigate its relations to the existing results in the literature. In order to do this we state an equivalent version of Theorem \ref{thm:main} in terms of $F$-contractions.

\begin{theorem}
	Let  $\psi\colon\mathbb{R}\to\mathbb{R}$ be  non-decreasing such that  $\lim_{n\to\infty} \psi^n(t)=-\infty$ for every $t\in(0,\infty)$  and $F\colon (0,\infty)\to \mathbb{R}$ be a function that satisfies one of the following two conditions:
	\begin{itemize}
		\item [(i)] $\inf F>-\infty$;
		\item [(ii)] for all $(\alpha_n)\subseteq(0,\infty)$ we have
		$
			\lim_{n\to\infty} F(\alpha_n)=-\infty$ iff $\lim_{n\to\infty}\alpha_n=0;
		$
	\end{itemize}
	Suppose that $(X,d)$ is a complete $b$-metric space with constant $s\ge1$ and $T\colon X\to X$ an operator such that for all $x,y\in X$
	\[ d(Tx,Ty)\ne 0 \mbox{ implies }F(d(Tx,Ty))\le \psi(F(d(x,y))).\]
	Then $T$ is a Picard-operator.
	\label{thm:main2}
\end{theorem}
Indeed, one can see that Theorem \ref{thm:main} is equivalent to Theorem \ref{thm:main2} by taking $G=e^F$ and $\varphi=\exp\circ\psi\circ\ln$.

\begin{remark}
	Theorem \ref{thm:main2} generalises Theorem 3.1 and Theorem 3.3 in \cite{secelean_wardowski_2016} by not requiring the boundedness conditions imposed on both $F$ and $T$ and by relaxing condition (F2).
\end{remark}
\begin{remark}
	By taking $\psi(t)=t-\tau$, for some $\tau>0$ in Theorem \ref{thm:main2}, we obtain a generalization of Theorem 2.1 in \cite{wardowski_fcontr_2012} by omitting conditions (F1) and (F3) on the function $F$ and by relaxing condition (F2).
\end{remark}

The following Corollary deals with the special case when $G$ is increasing.

\begin{corollary}  
	Let $(X,d)$ be a complete $b$-metric space with constant $s\ge1$, $G\colon (0,\infty)\to(0,\infty)$ be strictly increasing, $\varphi\in\Phi$ and $T\in\mathbb{T}(X,G,\varphi)$. Then $T$ is a Picard-operator.
	\label{cor:increasing}
\end{corollary}
\begin{proof}
	If $G$ is strictly increasing then for any sequence $(\alpha_n)\subseteq(0,\infty)$ we have
	\[
		\lim_{n\to\infty} G(\alpha_n)=0\mbox{ implies } \lim_{n\to\infty}\alpha_n=0.
	\]
	Indeed, let $\varepsilon>0$. It follows that there exists $n_\varepsilon\in\mathbb{N}$ such that $G(\alpha_n)<G(\varepsilon)$, for all $n\ge n_\varepsilon$, whence $\alpha_n<\varepsilon$, for all $n\ge n_\varepsilon$, thus $\lim_{n\to\infty}\alpha_n=0.$ 

	Now if $\inf G=0$ then, since $G$ is strictly increasing, we have
	\[
		\lim_{n\to\infty}\alpha_n=0\mbox{ implies } \lim_{n\to\infty} G(\alpha_n)=0,
	\]
	thus $G\in\mathbb{G}_1\cup\mathbb{G}_2$. Therefore the proof follows from Theorem \ref{thm:main}.
\end{proof}
\begin{remark}
	If we reformulate Corollary \ref{cor:increasing} in terms of $F=\ln G$ and $\psi(t)=\ln(\varphi(e^t))=t-\tau$ for some $\tau>0$, we obtain a generalization of Theorem 2.1 in \cite{wardowski_fcontr_2012} by omitting conditions (F2) and (F3). 
\end{remark}

\begin{remark}
    In the particular case when $G$ is increasing and invertible and $\varphi\in\Phi$ then
    $(G^{-1}\circ\varphi\circ G)\in\Phi$. Furthermore, the contractive condition in Theorem \ref{thm:main} is equivalent to
   \[
	d(Tx, Ty)\ne 0 \mbox{ implies }
    d(Tx, Ty)\le (G^{-1}\circ \varphi\circ  G)(d(x,y))
    \]
    for all $x,y\in X$. Thus Theorem \ref{thm:main} also generalises Theorem 1 in \cite{czerwik_1993}.
\end{remark}

\section{Proof of Theorem \ref{thm:main}}
For the sake of readability we extracted the distinguishable parts of the proof into a number of Lemmas.
\begin{lemma} If $\varphi\in\Phi$ then $\varphi(t)< t$, for all $t\in(0,\infty)$.
\label{lem:let}
\end{lemma}
\begin{proof}
	Suppose that there exists $t\in(0,\infty)$ such that $\varphi(t)\ge t$. Since $\varphi$ is non-decreasing we have
	\[\varphi^n(t)\ge\varphi^{n-1}(t)\ge\dots\ge\varphi(t)\ge t,\]
	for all $n\in\mathbb{N}$. Taking the limit when $n\to\infty$ implies
	\[0=\lim_{n\to\infty}\varphi^n(t)\ge t,\]
	which is a contradiction.
\end{proof}
First we focus on the case when $G\in \mathbb{G}_1$. In that case Theorem \ref{thm:main} is equivalent to the following Lemma.

\begin{lemma}Let $(X,d)$ be complete a $b$-metric space with constant $s\ge1$, $G\in\mathbb{G}_1$, $\varphi\in\Phi$ and $T\in\mathbb{T}(X,G,\varphi)$. Then $T$ is a Picard-operator.
\label{lem:g1}
\end{lemma}
\begin{proof}
	First we show that if there exists a fixed point of $T$, then it is unique. Indeed, if $x^*$ and $y^*\in X$ are such that $x^* = Tx^* \ne T y^*=y^*$, then
	\[
        0<G(d(x^*, y^*)) = G(d(Tx^*, Ty^*))\le \varphi(G(d(x^*, y^*))),
    \]
    which by Lemma \ref{lem:let} is a contradiction. 

	Now suppose that there exists an $x\in X$ such that $T^nx\ne T^{n+1}x$ for all $n\in\mathbb{N}$. Then
	\[ 
		0<G(d(T^nx,T^{n+1}x))\le \varphi^n(G(d(x,Tx))),
	\]
	for all $n\in\mathbb{N}$. Hence
	\[ 
		0\le\inf_n G(d(T^nx,T^{n+1}x))\le \inf_n \varphi^n(G(d(x,Tx)))=0,
	\]
	which contradicts $\inf G>0$. Consequently for every $x\in X$ there exists $n_0\in\mathbb{N}$ such that $T^{n_0+1}x=T^{n_0} x$, therefore $T^{n_0}x$ is a fixed-point of $T$.

\end{proof}

Now we deal with the case when $G\in\mathbb{G}_2$.
\begin{lemma}
	Suppose that $G\in\mathbb{G}_2$. Then there exists $\varepsilon_0>0$ such that for all $\varepsilon\in(0,\varepsilon_0)$ we have
	\begin{itemize}
		\item [(i)] $\sup_{\alpha<\varepsilon}G(\alpha)\le1$;
		\item [(ii)] $\inf_{\alpha\ge\varepsilon}G(\alpha)>0$.
	\end{itemize}
	\label{lem:infsup}
\end{lemma}
\begin{proof}
	Since $\lim_{n\to\infty}\alpha_n=0$ implies $\lim_{n\to\infty}G(\alpha_n)=0$ for every sequence $(\alpha_n)_{n\in\mathbb{N}}$, there exists an $\varepsilon_0>0$ such that $\alpha<\varepsilon_0$ implies $G(\alpha)\le1$. Hence if $\varepsilon\in(0,\varepsilon_0)$ then
	\[\sup_{\alpha<\varepsilon}G(\alpha)\le\sup_{\alpha<\varepsilon_0}G(\alpha)\le1.\]
	Suppose that $\inf_{\alpha>\varepsilon}G(\alpha)=0$. It follows that there exists a sequence $(\alpha_n)\subseteq(\varepsilon,\infty)$ such that $\lim_{n\to\infty}G(\alpha_n)=0$. Since $G\in\mathbb{G}_2$, this implies $\lim_{n\to\infty}\alpha_n=0$, which is a contradiction.
\end{proof}

\begin{lemma} Let $G\in \mathbb{G}_2$, $\varphi\in\Phi$, $T\in\mathbb{T}(X,G,\varphi)$ and construct $\varepsilon_0$ according to Lemma \ref{lem:infsup}. Then for all $\varepsilon\in(0,\varepsilon_0)$ there exists $n_\varepsilon\in\mathbb{N}$ such that 
\[d(x,y)<\varepsilon\mbox{ implies } d(T^nx,T^ny)<\frac{\varepsilon}{2s},\]
for all $n\ge n_\varepsilon$.
	\label{lem:neps}
\end{lemma}
\begin{proof}
	By Lemma \ref{lem:infsup} we know that $\sup_{\alpha<\varepsilon}G(\alpha)\le1$.
	Let $n_\varepsilon\in\mathbb{N}$ be such that 
	\[
		\inf_{\alpha>\frac{\varepsilon}{2s}}G(\alpha)>\varphi^{n_\varepsilon}(1).
	\] 
	For any $n\ge n_\varepsilon$ either we have $d(T^nx,T^ny)=0$ or the following chain of inequalities holds:
	\[G(d(T^nx,T^ny))\le \varphi^n(G(d(x,y)))\le \varphi^{n_\varepsilon}\left(\sup_{\alpha<\varepsilon}G(\alpha)\right)<\inf_{\alpha>\frac{\varepsilon}{2s}}G(\alpha).\]
	Hence in both cases we have $d(T^nx,T^ny)<\frac{\varepsilon}{2s}$, which finishes the proof.
\end{proof}

\begin{lemma}
	Let $G\in\mathbb{G}_2$, $\varphi\in\Phi$ and $T\in\mathbb{T}(X,G,\varphi)$. Then for all $\varepsilon>0$ and $n\in\mathbb{N}^*$ there exists $m_\varepsilon$ such that
	\[d(T^n x_{mn},x_{mn})<\frac{\varepsilon}{2s},\]
	for all $m\ge m_\varepsilon$.
	\label{lem:meps}
\end{lemma}

\begin{proof} 
	Let $n\in\mathbb{N}^*$. If there exists $m_0\in\mathbb{N}^*$ such that $d(T^n x_{m_0n},x_{m_0n})=0$ then for every $m\ge m_0$ we have $d(T^n x_{mn},x_{mn})=0$. Suppose that $d(T^n x_{mn},x_{mn})\ne0$ for all $m\in\mathbb{N}^*$. 
	Since $T\in\mathbb{T}(X,G,\varphi)$, we have
	\[
		0<G(d(T^n x_{mn},x_{mn}))\le \varphi^{mn}(G(d(T^n x_0,x_0))),
	\]
	for all $m,n\in\mathbb{N}^*$. 
	Let $m\to\infty$ to obtain $\lim_{m\to\infty}G(d(T^n x_{mn},x_{mn}))=0$, for all $n\in\mathbb{N}^*$. 
	Thus by $G\in\mathbb{G}_2$ we can conclude that $\lim_{m\to\infty}d(T^n x_{mn}, x_{mn})=0$, for all $n\in\mathbb{N}^*$. 	
\end{proof}

\begin{lemma} Let $G\in \mathbb{G}_2$, $T\in\mathbb{T}(X,G)$ and construct $\varepsilon_0$ according to Lemma \ref{lem:infsup}. Then for every $\varepsilon\in(0,\varepsilon_0)$ there exists $n=n_\varepsilon, m=m_\varepsilon\in\mathbb{N}^*$ such that
\[T^n(B(x_{mn},\varepsilon))\subseteq B(x_{mn},\varepsilon),\]
where $B(x_{mn},\varepsilon)=\{u\in X\mid d(u,x_{mn})<\varepsilon\}$.	
\label{lem:nmeps}
\end{lemma}

\begin{proof}
	By Lemma \ref{lem:neps} there exists $n=n_\varepsilon$ such that 
	\[d(x,y)<\varepsilon \mbox{ implies } d(T^n x, T^n y)\le \frac{\varepsilon}{2s}.\]
	We can use Lemma \ref{lem:meps} for this $n$ to obtain an $m=m_\varepsilon$ satisfying
	\[d(T^n x_{mn},x_{mn})<\frac{\varepsilon}{2s}.\]
	Now let $u\in B(x_{mn},\varepsilon)$. By the above,
	\[d(T^n u, x_{mn})\le s\cdot d(T^n u, T^n x_{mn})+s\cdot d(T^n x_{mn},x_{mn})<s\cdot \frac{\varepsilon}{2s}+s\cdot \frac{\varepsilon}{2s}=\varepsilon.\]
\end{proof}

\begin{lemma} Let $(x_k)_{k\in\mathbb{N}}$ be a sequence such that $d(x_{k+1},x_k)\to0$. Then for all $n\in \mathbb{N}^*$ and $\varepsilon>0$ there exists $m_0\in\mathbb{N}$ such that
\[d(x_{mn},x_{mn+p})<\varepsilon,\]
for all $m\ge m_0$ and $p\in\{0,\dots,n-1\}$.
\label{lem:steps}
\end{lemma}
\begin{proof}
	Let $n\in\mathbb{N}^*$ and $\varepsilon>0$. Since $d(x_{k+1},x_k)\to0$, there exists $k_0\in \mathbb{N}$ such that 
	\[d(x_{k+1},x_k)<\frac{\varepsilon}{ns^n},\]
	for all $k\ge k_0$.
	Let $m_0$ be such that $m_0n>k_0$. We have
	\[
		d(x_{mn},x_{mn+p})\le\sum_{i=0}^p  s^{i+1} d(x_{mn+i},x_{mn+i+1})\le\sum_{i=0}^{n-1}  s^n\frac{\varepsilon}{ns^n}=\varepsilon,
	\]
	for all $m>m_0$ and $p\in\{0,...,n-1\}$, concluding the proof.
\end{proof}

\begin{pr}
	The uniqueness of the fixed-point can be shown with the technique used in the proof of Lemma \ref{lem:g1}.

	In the following we prove that the sequence  $(x_k)_{k\in\mathbb{N}}$ is Cauchy. Let $\varepsilon>0$ be arbitrary. We can suppose that $\varepsilon<\varepsilon_0$ (where $\varepsilon_0$ is defined as in Lemma \ref{lem:infsup})  since in the other case $d(x_{k_1},x_{k_2})< \varepsilon_0$ implies $d(x_{k_1},x_{k_2})< \varepsilon$.

	By Lemma \ref{lem:nmeps} there exists $n,m\in\mathbb{N}^*$ such that 
	\[T^n(B(x_{mn},\varepsilon))\subseteq B(x_{mn},\varepsilon).\]
	Since $d(x_{k+1},x_k)\to 0$ we can construct an $m_0$ as in Lemma \ref{lem:steps}. Now let $\overline m=\max\{m,m_0\}$ and $k_1,k_2\in \mathbb{N}$ such that $k_1,k_2\ge \overline m n$.
	We can write $k_1=m_1n+p_1$, $k_2=m_2n+p_2$, where $p_1,p_2\in\{0,...,n-1\}$ and $m_1,m_2\ge \overline m$. 

	The construction of these indices implies
	\begin{align*}
		d(x_{k_1},x_{m_1n})<\varepsilon \mbox{ and } d(x_{m_2n},x_{k_2})<\varepsilon,&\quad\mbox{by Lemma \ref{lem:steps}};\\
		d(x_{m_1n},x_{mn})<\varepsilon\mbox{ and } d(x_{mn},x_{m_2n})<\varepsilon,&\quad\mbox{by Lemma \ref{lem:nmeps}}.
	\end{align*} 
	Therefore we can use the triangle inequality to obtain
	\begin{align*}
		d(x_{k_1},x_{k_2})&\le s d(x_{k_1},x_{m_1n})+s^2d(x_{m_1n},x_{mn})+s^3d(x_{mn},x_{m_2n})+s^3d(x_{m_2n},x_{k_2})\\
		&\le s\varepsilon+s^2\varepsilon+s^3\varepsilon+s^3\varepsilon\le 4s^3\varepsilon.
	\end{align*}

	Thus we proved that $(x_k)_{k\in\mathbb{N}}$ is a Cauchy sequence.
	Since $(X, d)$ is complete, there exists an $x^*\in X$ such that $x_{k} \to x^*$. Then by Lemma \ref{lem:liminfsup} 
    \begin{align*}
    s^{-1} d(x^*, Tx^*)&\le \liminf_{k\to\infty} d(x_{k+1}, T x^*)\le \limsup_{k\to \infty} d(x_{k+1}, T x^*) \\
    &= \limsup_{k\to\infty} d(T x_{k}, T x^*) \le \limsup_{k\to \infty} d(x_{k}, x^*)= 0,
    \end{align*}
    hence $T x^* = x^*$.
\end{pr}
\bibliographystyle{plain}
\bibliography{myref}{}
\end{document}